\documentclass[a4paper,10pt,reqno,onesided]{amsart}

\usepackage{amsmath,amsfonts,amsthm,amssymb,dsfont,mathabx}
\usepackage[OT4]{fontenc}
\usepackage{enumerate}
\usepackage{mathrsfs,mathtools}
\usepackage{enumerate,enumitem, xspace}
\usepackage[pdftex]{graphicx}
\usepackage{color}
\usepackage{float}
\usepackage[center, font=small,labelfont=sc,labelsep=none]{caption}
\usepackage{euscript}
\usepackage{setspace}
\usepackage{fancyhdr}
\usepackage{url}
\usepackage{algorithm} 
\usepackage{algpseudocode} 
\usepackage{thmtools}
\usepackage[colorlinks=true,allcolors=blue,backref=page]{hyperref}
\usepackage[noabbrev,capitalize,nameinlink]{cleveref}

\definecolor{darkred}{RGB}{200,0,0}
\hypersetup{colorlinks={true},linkcolor={darkred},citecolor={darkred}}

\usepackage[pagewise]{lineno}

\long\def\symbolfootnote[#1]#2{\begingroup
\def\thefootnote{\fnsymbol{footnote}}\footnote[#1]{#2}\endgroup}

\newtheorem{theorem}{Theorem}[section]
\newtheorem{lemma}[theorem]{Lemma}

\newtheorem{cor}[theorem]{Corollary}

\newtheorem*{question*}{Question}
\theoremstyle{remark}

\newtheorem*{claim*}{Claim}
\newtheorem{remark}[theorem]{Remark}

\newtheorem{example}[theorem]{Example}
\newtheorem*{remark*}{Remark}
\newtheorem*{example*}{Example}

\newtheorem*{fact*}{Fact}

\newtheoremstyle{named}{}{}{\itshape}{}{\bfseries}{.}{.5em}{\thmnote{#3}}
\theoremstyle{named}
\newtheorem*{namedtheorem}{Theorem}
\newtheorem*{namedlemma}{Lemma}
\theoremstyle{plain}

\theoremstyle{definition}

\newcommand{\executeiffilenewer}[3]{%
\ifnum\pdfstrcmp{\pdffilemoddate{#1}}%
{\pdffilemoddate{#2}}>0%
{\immediate\write18{#3}}\fi%
}

\begin{document}

\title{Books, Hallways and Social Butterflies:\\ A Note on Sliding Block Puzzles}

\author[F.~Brunck]{Florestan Brunck}
\author[M.~Kwan]{Matthew Kwan}

\address{Institute of Science and Technology Austria\\
 Am Campus 1, 3400, Klosterneuburg, Lower Austria, Austria}
\email{florestan.brunck@ist.ac.at}
\email{matthew.kwan@ist.ac.at}

\setlength{\headheight}{2em}
\setlength{\footskip}{2.5em}
\fancyhf{}
\fancyfoot[C]{\thepage}
\fancyhead[CO]{\textsc{f.~brunck} \& \textsc{m.~kwan}}
\fancyhead[CE]{\textsc{\small{books, hallways and social butterflies: a note on sliding block puzzles}}}
\pagestyle{fancy}

\maketitle

\begin{abstract}

\noindent Recall the classical \emph{15-puzzle}, consisting of 15 sliding blocks in a $4\times 4$ grid. Famously, the configuration space of this puzzle consists of two connected components, corresponding to the odd and even permutations of the symmetric group $S_{15}$. In 1974, Wilson generalised sliding block puzzles beyond the $4\times 4$ grid to arbitrary graphs (considering $n-1$ sliding blocks on a graph  with $n$ vertices), and characterised the graphs for which the corresponding configuration space is connected.

In this work, we extend Wilson's characterisation to sliding block puzzles with an arbitrary number of blocks (potentially leaving more than one empty vertex). For any graph, we determine how many empty vertices are necessary to connect the corresponding configuration space, and more generally we provide an algorithm to determine whether any two configurations are connected. Our work may also be interpreted within the framework of \emph{Friends and Strangers graphs}, where empty vertices correspond to ``social butterflies'' and sliding blocks to ``asocial'' people.
\end{abstract}

\section{Introduction}
\label{intro}

Interest in sliding block puzzles dates back to the \emph{15-puzzle}, seemingly invented by Noyes Chapman in 1874 (see \cite{15} for an account of the fascinating history of the puzzle). The game consists of fifteen movable square blocks numbered $1,2,\ldots, 15$ and arranged within a 4 by 4 square box, leaving one empty space (see \cref{fig:15}). The task at hand is to start from a given configuration of the numbered blocks and reach the desired target configuration, where the only allowed move is to slide a numbered block into an adjacent empty space. This task seemed to be unpredictably either very easy to accomplish, or completely impossible, and the puzzle turned into a worldwide sensation in the spring of 1880. A particularly challenging instance, known as the ``13-15-14 puzzle'', is where the initial and target configuration differ by a single swap (historically this swap involved the blocks labelled 14 and 15). The craze of this puzzle was such that it consistently made newspaper headlines in 1880, with an article of the \emph{New York Times} lamenting it was ``threatening our free institutions'' (\cite[p9]{15}). Various prizes were offered for anyone who would solve this challenge at last --- starting with a \$25 set of teeth and culminating with Sam Loyd's famous \$1,000 cash prize.

\begin{figure}[H]\centering
  \includegraphics[page=8]{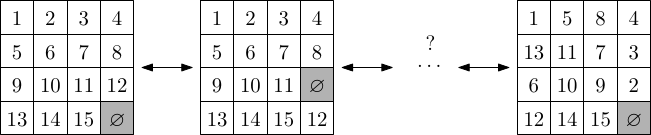}
  \caption{: Two configurations of the 15 puzzle are not necessarily connected by elementary moves, depending on the parity of the initial and target configurations.}
  \label{fig:15}
\end{figure}
\vspace{-1em}

The notorious 13-15-14 challenge is now known to be impossible. The earliest mathematical accounts were published shortly after the puzzle became a sensation. Johnson and Storey (in tandem)~\cite{story} and Schubert~\cite{Schubert} independently realised that each configuration can be assigned a parity, and only configurations of the same parity can be reached from each other.

It is natural to wonder how much of this analysis is dependent on the particular setup of the 15-puzzle. Indeed, over the years, many versions and extensions of the 15-puzzle have been proposed and studied, drawing the attention of those in both recreational and academic communities (including Knuth and Coxeter). See for example \cite{Archer, Coxeter, Spitznagel,Wilson,Yang,Efficient,Varikon,Twist,Permutation}, \cite[Chapter~9 and pp.\ 114--115 of Chapter~7]{15} and \cite[B60, E19 and E20]{hordern}. In \cref{fig:coiled-15} we single out a particular variation on the 15-puzzle, which we shall reuse later on in a few examples.

\begin{figure}[H]\centering
  \includegraphics[page=35]{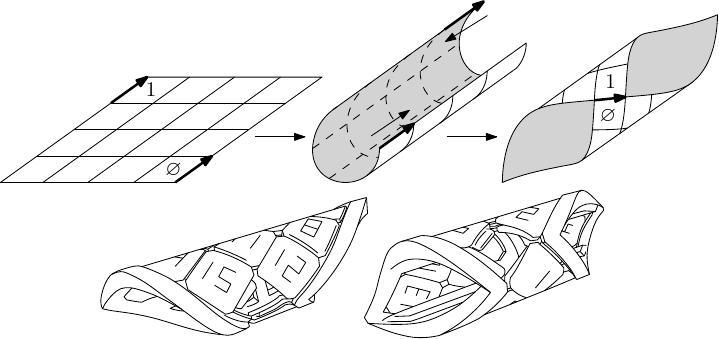}
  \caption{: Henry Segerman's \emph{coiled 15} wraps around the standard $15$-puzzle on a cylinder, connecting the top left square with the bottom right square.}
  \label{fig:coiled-15}
\end{figure}
\vspace{-1em}

Perhaps most famously, in 1974 Wilson~\cite{Wilson} generalised the 15-puzzle to the setting where blocks may be slid around on an \emph{arbitrary graph}. Specifically, we can consider a graph $\mathcal G$ and place a numbered block on all but one of the vertices. Then, the only legal move is to move a block onto an empty adjacent vertex. Wilson gave a complete characterisation of the graphs $\mathcal G$ for which the corresponding configuration space is connected (in the language of graph theory, $\mathcal G$ should be \emph{biconnected} and \emph{non-bipartite}, and not equal to a particular 7-vertex graph called $\theta_0$, see \cref{fig:theta}) 

Wilson's theorem is not entirely the end of the story. As we have discussed, the 15-puzzle does not have a connected configuration space (its corresponding graph $G$ is bipartite). However, if we remove an additional block, leaving \emph{two} empty spaces, then it is easy to show that the configuration space becomes connected. In general, for a given graph it is unclear exactly how much more power is granted by the addition of more empty spaces. In this paper, we set out to understand the full picture for sliding block puzzles in general graphs, with an arbitrary number of empty spaces (see for example \cref{fig:black-and-white}). In particular, we provide a complete answer to the following question.

\begin{figure}[H]\centering
  \includegraphics[page=43]{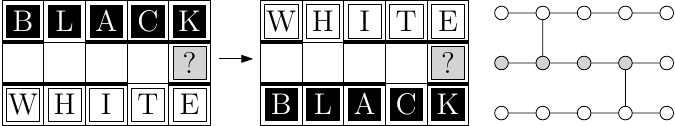}
  \caption{: Nob Yoshigahara's \emph{Black and White} puzzle (1982) is an example of a sliding block puzzle with more than one empty space. The player must interchange the black and white tiles, respecting the order of the letters. Note the immovable barriers between the rows in bold. The corresponding underlying graph $\mathcal G$ is depicted on the right, with the 4 empty spaces highlighted in grey.}
  \label{fig:black-and-white}
\end{figure}
\vspace{-1em}

\begin{question*}
Given an arbitrary graph $\mathcal G$, how many empty vertices are necessary to connect all the configurations of its associated sliding block puzzle?
\end{question*}

\begin{remark}\theoremstyle{remark}
\label{remark-k}
Unsurprisingly, connectedness of the configuration space is a property that is monotone in $k$: if $k$ empty vertices are enough to connect all configurations, any number $k'>k$ of empty vertices is certainly also enough to achieve connectivity. Although intuitively clear, this remark can be seen as a particular case of \cite[Proposition~2.1]{F&S}.
\end{remark}

\begin{remark}
Since making this paper publicly available, we were made aware that the answer to the question above has previously appeared in the literature. Specifically, this result (in different language) is claimed without proof in a paper of Kornhauser, Miller and Spirakis~\cite{Pebbles}, appearing in the conference proceedings of FOCS'84 (which is a conference in theoretical computer science; there are connections to memory management in distributed systems). The proof can be found in Kornhauser's thesis \cite{Thesis}. Nonetheless, we believe our new proof provides a different outlook on this theorem, and we hope that our reframing within the language of Friends and Strangers Graphs might help bring attention to this result from different communities. We also provide a different decision algorithm to check if two puzzles can be reconfigured into each other (see \cref{connected components}, simplifying in our view the work of \cite{Thesis}), together with a new result giving an exact count of the number of connected components of the space of all configurations (\cref{sec:structure}).
\end{remark}

\subsection*{Friends and Strangers Graphs.} Recently, Defant and Kravitz situated sliding block puzzles as a special case of the broader framework of ``friends and strangers graphs''\cite{F&S} (see also \cite{F&S-1,F&S-2,F&S-3,F&S-4,F&S-5,F&S-6,F&S-7,F&S-8}), which we shall recall and adopt in the rest of this article.

Informally, in this new framework, one may think of sliding blocks as people walking on an underlying simple graph $\mathcal G$ (e.g., the 4 by 4 grid graph in the case of the 15 puzzle). Whenever two people find themselves at the endpoints of a common edge of $\mathcal G$, they may swap their position/vertex in $\mathcal G$ if and only if they are friends with one another. In this context, an empty space can be seen as a ``social butterfly'', i.e.\ a person who is friends with everyone else. The friendships are specified using a second graph $\mathcal F$, in which each vertex corresponds to a fixed person, and edges indicate friendships (see \cref{fig:example-15}). In our notation, we make use of capital letters to denote the positions of the people (i.e.\ the vertices of $\mathcal G$) and lowercase letters to refer to the people themselves (i.e.\ the vertices of $\mathcal F$),  and say that a person $x\in\mathtt V(\mathcal F)$ \emph{lies} on a vertex $X\in \mathtt V(\mathcal G)$.

Formally, given two simple graphs $\mathcal G$ and $\mathcal F$ on $n$ vertices, the \textit{friends-and-strangers graph} $\operatorname{FS}(\mathcal G, \mathcal F)$ is the graph whose vertex set is the set of bijections $\sigma : \mathtt V(\mathcal G) \rightarrow \mathtt V(\mathcal F)$, where $\mathtt V(\mathcal G)$ denotes the vertex set of a graph $\mathcal G$. We call such a bijection a \textit{configuration}. Two configurations $\sigma, \sigma ' : \mathtt V(\mathcal G) \rightarrow \mathtt V(\mathcal F)$ are then adjacent in $\operatorname{FS}(\mathcal G, \mathcal F)$ if and only if there exists an edge $\{A, B\}$ in $\mathcal G$ such that:

\begin{itemize}[label=\raisebox{0.25ex}{\tiny$\bullet$}]
 \item  $\{\sigma(A), \sigma(B)\}$ is an edge in $\mathcal F$.
 \item  $\sigma(A) = \sigma'(B)$ and $\sigma(B) = \sigma'(A)$.
 \item $\sigma(C) = \sigma'(C)$, for all $C \in \mathtt V (\mathcal G) - \{A, B\}$.
\end{itemize}

When this is the case, the operation that transforms $\sigma$ into $\sigma'$ is referred to as a \textit{$(\mathcal G, \mathcal F)$-friendly swap across} $\{A, B\}$ in \cite{F&S}, or simply a \textit{swap} across $\{A, B\}$. 

A vertex $v\in \mathtt V (\mathcal F)$ with degree $n-1$ is called a \textit{social butterfly}. On the other hand, a vertex $v\in \mathtt V (\mathcal F)$ who is friends with nobody except social butterflies is called \textit{asocial}. 

\begin{figure}[H]\centering
  \includegraphics[page=28]{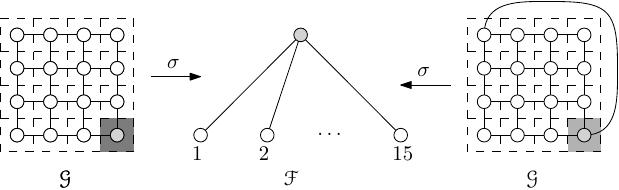}
  \vspace{-1em}
  \caption{: The friendship graph $\mathcal F$ associated to the $15$-puzzle (left) and the coiled $15$-puzzle (right) is the same star graph $\mathcal K_{1,15}$ on $16$ vertices (middle). For each of these puzzles, a configuration is encoded as a bijection from the vertices of $\mathcal G$ to those of $\mathcal F=\mathcal K_{1,15}$.}
  \label{fig:example-15}
\end{figure}
\vspace{-1em}

\begin{example}
\label{example-15}
The 15-puzzle can be encoded by setting $\mathcal G=\operatorname{Path}_4\square \operatorname{Path}_4$ and setting $\mathcal F$ to be the complete bipartite graph $\mathcal K_{1,15}$ (i.e.\ the star graph with $n-1$ edges). In this example, the centre of the star is the unique social butterfly and everyone else is asocial. See \cref{fig:example-15}.
\end{example}

\begin{example}
\label{example-coil}
The coiled 15-puzzle can be encoded using the same friendship graph $\mathcal F=\mathcal K_{1,15}$ and setting $\mathcal G$ to be the $4\times 4$ grid graph $\operatorname{Path}_4\square \operatorname{Path}_4$ to which a single edge joining the top-left and bottom-right vertex is added (for example). See \cref{fig:example-15}.
\end{example}

The central question and results of this article may then be understood under two different guises: either as the reconfiguration of sliding block puzzles using empty spaces, or as the reconfiguration of an asocial crowd walking on a graph, with the help of social butterflies.

In the case of the 15-puzzle --- and more generally in Wilson's paradigm --- the friendship graph $\mathcal F$ turns out to be a star graph (as seen in \cref{example-15} and \cref{example-coil}). When considering multiple social butterflies, the right generalisation of star graphs is that of \textit{generalised book graphs}, as defined for example in \cite{Book}. Given integers $n>k$, we consider the $k$-\textit{book graph on $n$ vertices} $\mathcal B_{n-k}^{(k)}$, consisting of $n-k$ copies of the complete graph $\mathcal K_{k+1}$ on $k+1$ vertices, joined along a common $\mathcal K_k$ (see \cref{fig:book}). 

\begin{figure}[H]\centering
  \includegraphics[page=2]{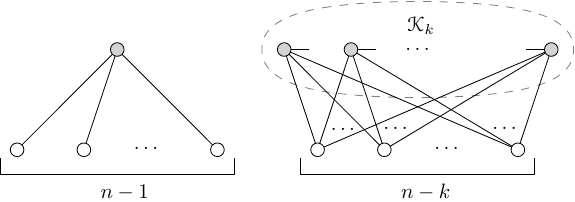}
  \caption{: On the left, the star graph $\mathcal K_{1,n-1}$, alternatively seen as the book graph $\mathcal B_{n-1}^{(1)}$, with its single common vertex $\mathcal K_1$ highlited in grey. On the right, a general $k$-book graph on $n$ vertices, with its common $\mathcal K_k$ indicated in grey.}
  \label{fig:book}
\end{figure}
\vspace{-1em}

\begin{remark*}
Using our terminology, book graphs can be seen as the family of graphs in which every person is either asocial or a social butterfly.
\end{remark*}

We may now reformulate our main question in the language of the friends and strangers graphs:

\begin{question*}[Friends and Strangers reformulation]
Given an arbitrary graph $\mathcal G$ on $n$ vertices, for which values of $k$ is $\operatorname{FS}(\mathcal G,\mathcal B_{n-k}^{(k)})$ connected? In other words, what is the minimal number of social butterflies required to freely reconfigure all the people on $\mathcal G$ if everyone else is asocial?
\end{question*}

 \begin{figure}[H]\centering
  \includegraphics[page=15]{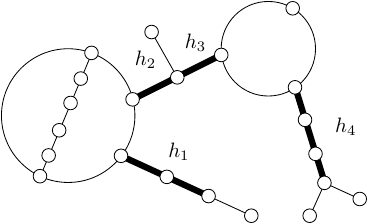}
  \caption{: A graph with $4$ maximal hallways $h_1$, $h_2$, $h_3$ and $h_4$ (highlighted in bold) of order  $3$, $2$, $2$, and $4$ respectively.}
  \label{fig:hallways}
\end{figure}
\vspace{-1em}

 To formulate our Main Theorem, which precisely answers this last question, we first need to define the notion of a \textit{hallway}.  Recall first that a \textit{path} $\mathcal P$ in the graph $\mathcal G$ is a sequence $\mathcal P=(X_0, X_1, \ldots, X_n)$ of distinct vertices of $\mathcal G$ such that, for all $1 \leq i\le n$, $X_{i-1}$ and $X_i$ are adjacent in $\mathcal G$. The path $\mathcal P$ is said to be \emph{bare} if all its inner vertices have degree 2 in $\mathcal G$. A vertex $V\in \mathtt V(\mathcal G)$ is a \textit{cut-vertex} of $\mathcal G$ if removing it increases the number of connected components of $\mathcal G$. A \emph{hallway} of \emph{order} $k$ in $\mathcal G$  is a bare path $\mathcal H=(X_1, X_2, \ldots, X_k)$ in $\mathcal G$ between cut-vertices $X_1$ and $X_k$ that is not part of any cycle (in particular, a single cut-vertex is a hallway of order 1). See \cref{fig:hallways}. A hallway is called \emph{dark} if its order is greater than or equal to $k$ (the relevant value of $k$ will usually be clear from the context).

 \begin{figure}[H]\centering
  \includegraphics[page=16]{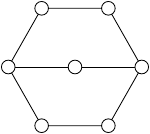}
  \caption{: The $\theta_0$ graph.}
  \label{fig:theta}
\end{figure}
\vspace{-1em}

Let us denote by $\kappa(\mathcal G)$ the maximal order of a hallway in $\mathcal G$. We then define the function $\kappa^*:\mathcal G\rightarrow \{1, 2, \ldots, n\} \cup \{\infty\}$ to be equal to $\kappa(\mathcal G)+1$, except in the following cases: 

\begin{itemize}[label=\raisebox{0.25ex}{\tiny$\bullet$}]
\item If $\mathcal G$ is not connected, then $\kappa^*(\mathcal G)=\infty$.
\item If $\mathcal G$ is a cycle, then $\kappa^*(\mathcal G)=n-1$.
\item If $\mathcal G$ is either bipartite or the $\theta_0$ graph (defined in \cref{fig:theta}), then $\kappa^*(\mathcal G)=2$.
\end{itemize}

\begin{namedtheorem}[Main Theorem]
\label{thm:main}
Let $\mathcal G$ be an arbitrary graph on $n$ vertices and let $\mathcal B_{n-k}^{(k)}$ be the $k$-book graph on $n$ vertices. The smallest $k$ such that $\operatorname{FS}(\mathcal G,\mathcal B_{n-k}^{(k)})$ is connected is exactly $\kappa^*(\mathcal G)$. 
\end{namedtheorem}

In other words, our \nameref{thm:main} states that the only obstructions to a connected configuration space are cycles, bipartite graphs, the theta graph $\theta_0$, and dark hallways.

Our \nameref{thm:main} can actually also be used to derive corollaries about the connected components of $\operatorname{FS}(\mathcal G,\mathcal B_{n-k}^{(k)})$ for general $\mathcal G$ and $k$ (i.e.\ when $\operatorname{FS}(\mathcal G,\mathcal B_{n-k}^{(k)})$ may not be connected). In particular, we describe an efficient algorithm for determining whether two configurations can be reached from each other, and we provide a general formula for the number of connected components of $\operatorname{FS}(\mathcal G,\mathcal B_{n-k}^{(k)})$. See \cref{connected components,sec:structure}.

\section{Overview of the Proof}
\label{overview}

Before giving an informal overview of the proof of the \nameref{thm:main}, we begin by recalling the situation for sliding block puzzles with exactly one empty vertex. Wilson's theorem~\cite{Wilson} (together with some observations due to Defant and Kravitz~\cite{F&S}) gives a description of the component structure of $\operatorname{FS}(\mathcal G, \mathcal K_{1,n-1})$, as follows.
Recall that a graph is \emph{biconnected} if it has no cut-vertex.

\begin{namedtheorem}[Wilson's Theorem]
\label{thm:Wilson}
Let $\mathcal G$ be a graph on $n$ vertices. We have the following cases.
\begin{enumerate}
\setcounter{enumi}{-1}
    \item \textit{$n\le 2$}. In this case, $\operatorname{FS}(\mathcal G, \mathcal K_{1,n-1})$ is trivially connected.
    \item \textit{$\mathcal G$ is biconnected, not bipartite, not a cycle, and not $\theta_0$}. In this case $\operatorname{FS}(\mathcal G, \mathcal K_{1,n-1})$ is connected.
    \item \textit{$\mathcal G$ is biconnected and bipartite, has at least 3 vertices, and is not a cycle}. In this case, $\operatorname{FS}(\mathcal G, \mathcal K_{1,n-1})$ has exactly two connected components of equal size. 
    \item \textit{$\mathcal G$ is not biconnected}. In this case, $\operatorname{FS}(\mathcal G, \mathcal K_{1,n-1})$ is disconnected.
    \item \textit{$\mathcal G$ is a cycle.} In this case, $\operatorname{FS}(\mathcal G, \mathcal K_{1,n-1})$ has as many connected components as cyclic orders of $\{1,\dots,n\}$ (i.e.\ $n!/2$ components of size $n(n-1)$).
    \item \textit{$\mathcal G$ is the $\theta_0$ graph.} $\operatorname{FS}(\theta_0, \mathcal K_{1,n-1})$ is not connected, and one can obtain a full description of the $6$ connected components and their structure (see \cite{Wilson} for the details).
\end{enumerate}

Moreover, in the setting of (2), we have the following characterisation of the components of $\operatorname{FS}(\mathcal G, \mathcal K_{1,n-1})$. ~Under some chosen identification of both $\mathtt V(\mathcal G)$ and $\mathtt V(\mathcal K_{1,n-1})$ with $\{1,2,\ldots, n\}$, configurations can be seen as permutations of $\{1,2,\ldots, n\}$. Denoting by $S\in\mathtt V(\mathcal G)$ the vertex supporting the unique social butterfly in the identity permutation/configuration, the two components of $\operatorname{FS}(\mathcal G, \mathcal K_{1,n-1})$ correspond exactly to the permutations whose parity are the same (resp. the opposite) as the parity of the distance in $\mathcal G$ between the vertex supporting their social butterfly and $S$.
\end{namedtheorem}

To be precise about attribution, Wilson's 1974 paper~\cite{Wilson} proves (1), (2) and (5), together with the characterisation for (2) described at the end of the theorem statement (these are all part of \cite[Theorem~1]{F&S}). Defant and Kravitz~\cite{F&S} observe (4) (see the discussion preceding \cite[Theorem 2.5]{F&S}), and in the setting of (3) they prove a lower bound for the number of connected components in terms of the number of cut-vertices (see \cite[Proposition~2.6]{F&S}). For the exact number of connected components in the setting of (3), we refer the reader to our more general \cref{formula}, which relies on definitions not yet introduced.

 \begin{figure}[H]\centering
  \includegraphics[page=29]{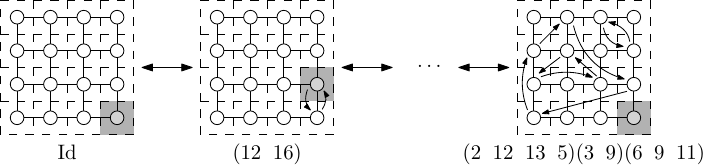}
  \caption{: \nameref{thm:Wilson} gives a simple criterion to check if two configurations are connected when only $1$ empty vertex is involved.
  }
  \label{fig:wilson}
\end{figure}
\vspace{-1em}

\begin{example}
The first two configurations of the $15$-puzzle shown on \cref{fig:15} are trivially connected as they differ by a single flip. Alternatively, this can be checked using \nameref{thm:Wilson}: if we work under the identification where the first configuration corresponds to the identity permutation (where $16$ is mapped to the empty vertex), the second configuration then corresponds to the transposition $(12 \;\; 16)$, see \cref{fig:wilson}. This permutation is odd, but so is the distance between the empty spaces of the first and second configurations and they are thus in the same connected component. Likewise, the last configuration shown in \cref{fig:15} is $(2\;\;12\;\;13\;\;5)(3\;\;9)(6\;\;9\;\;11)$, which is even. Since the empty spaces of the first and last configurations are in the same position (and $0$ is even), these configurations are in fact also connected.
\end{example}

\begin{example}
Because the underlying graph of the coiled-$15$ (see \cref{fig:coiled-15}) is no longer bipartite, case $(1)$ of \nameref{thm:Wilson} applies and all configurations become connected.
\end{example}

 \begin{figure}[H]\centering
  \includegraphics[page=17]{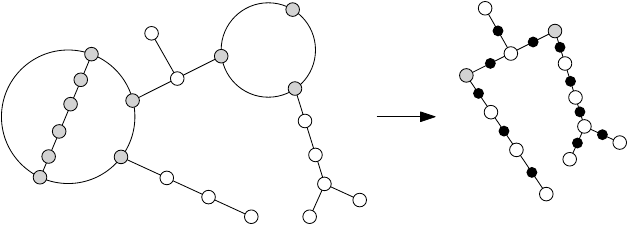}
  \caption{: A graph (left) and its associated block graph (right), with its cut-vertices in black, 2 libraries in grey and all other blocks in white.}
  \label{fig:blocks}
\end{figure}
\vspace{-1em}

In order to understand the situation when more social butterflies are involved, we first recall that a \emph{block} in a graph is a maximal biconnected subgraph, and we recall that every connected graph $\mathcal G$ admits a canonical tree-like decomposition into its blocks (see for example \cite[Section~3.1]{Diestel}). More formally, let us denote the set of cut-vertices of $\mathcal G$ by $\mathcal A$ and the set of blocks of $\mathcal G$ by $\mathfrak B$. The \textit{block graph} is defined to be the graph with vertex set $\mathcal A\cup \mathfrak B$ and edge set $\{A \mathcal B \> | \mathcal B\in \mathfrak B, A\in \mathcal B\}$. The block graph of a connected graph $\mathcal G$ is a tree. We will further distinguish between blocks consisting of a single edge and biconnected blocks on at least 3 vertices, the latter of which we shall refer to as \emph{libraries} (see \cref{fig:blocks}). 

In order to prove our \nameref{thm:main}, we provide a variation of \nameref{thm:Wilson} which allows us to fully understand the situation on each library (see the \nameref{balloon} in the next section). We then show how one can use the block decomposition of a graph to inductively bootstrap this strengthening and obtain an analysis of the whole graph from the individual analysis of each of its libraries. The intuition for this latter step stems from the observation that hallways play a crucial role in ferrying people across the block graph and in between libraries. Namely: if a hallway is longer than the number of social butterflies available (i.e.\ if the hallway is \emph{dark}), we cannot ferry any asocial people across it, as follows.

\begin{lemma}
\label{DF-hallways}
Fix any $k\ge 1$. If $\mathcal G$ is a graph with a dark hallway, then $\operatorname{FS}(\mathcal G,\mathcal B_{n-k}^{(k)} )$ is disconnected.
\end{lemma}

\cref{DF-hallways} is easy to see directly, but can also be seen as a corollary of \cite[Theorem 6.1]{F&S}, which more generally states that $\operatorname{FS}(\mathcal G,\mathcal F )$ is disconnected whenever $\mathcal G$ has a dark hallway and $\mathcal F$ has minimum degree at most $k$.

To establish our results, we shall heavily rely on the connectedness of $\operatorname{FS}(\mathcal G, \mathcal B_{n-k}^{(k)})$ for a few special types of graphs $\mathcal G$ which we analyse in \cref{sec:special-graphs}, in a series of lemmas of increasing generality.

\section{the snake tongue, the stopwatch, the hourglass and the balloon}\label{sec:special-graphs}
\label{lemmas}

In this section, we establish connectivity of $\operatorname{FS}(\mathcal G, \mathcal B_{n-k}^{(k)})$ for some special graphs $\mathcal G$. First we prove the \nameref{snaketongue}, then a series of lemmas culminating in the \nameref{balloon}, taking a short detour for the \nameref{hourglass}.

To give a quick summary of how everything comes together: the proof of our \nameref{thm:main} will be by induction; the \nameref{snaketongue} and the \nameref{balloon} are used in the inductive step (in particular, the \nameref{snaketongue} will be used to ferry people across hallways), while the base cases of the induction are provided by the \nameref{stringless} (proved on the way to the \nameref{balloon}) and the \nameref{hourglass}.

We point out to the reader that, in all the following lemmas, we implicitly make use of \cref{remark-k} in our proofs, which states that we need only provide each proof for the smallest relevant $k$ (usually $k=2$). 

Our first special graph is the \emph{snake tongue} graph $\mathcal D_n$, $n\geq 4$, defined as the Dynkin diagram of type $\mathcal D_n$, i.e.\ the graph on $n$ vertices obtained from a path on $n-1$ vertices by adding a leaf to the second-last vertex of the path (see \cref{fig:snake}). 

\begin{remark*}
From now on, in all our figures, we shall indicate vertices occupied by social butterflies in grey.
\end{remark*}

 \begin{figure}[H]\centering
  \includegraphics[page=19]{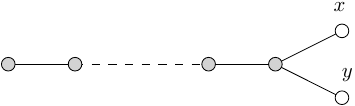}
  \caption{: The snake tongue graph $\mathcal D_n$.
  }
  \label{fig:snake}
\end{figure}
\vspace{-1em}

\begin{namedlemma}[Snake Tongue Lemma]
\label{snaketongue}
For all $n\in\mathbb N$, $n-2\leq k 
\leq n$, $\operatorname{FS}(\mathcal D_n,\mathcal B_{n-k}^{(k)})$ is connected.
\end{namedlemma}
In words, the \nameref{snaketongue} says that $\mathcal D_n$ can be reconfigured with $n-2$ social butterflies.

\begin{proof}
This is trivial if $k\ge n-1$, because then $B_{n-k}^{(k)}$ is a clique and everyone can move around freely. If $k=n-2$ there are only two asocial people, who can exchange their positions by moving to the ``tip of the snake tongue'', then moving back in the opposite order. (Formally, this is a special case of \cite[Theorem~6.5]{F&S}, since $\mathcal B_{n-k}^{(k)}$ has minimum degree at least $n-2$).
\end{proof}

We next define the \emph{stopwatch} graph $\mathcal S\mathcal W_n$, $n\geq 2$, as a cycle $\mathcal C_{n-1}$ of length $n-1$ (the \emph{dial} of the stopwatch) with an added leaf (the \emph{crown} of the stopwatch), see \cref{fig:clockwork}. We denote by $O$ and $P$ the unique vertices of degree 1 and 3, respectively, and write $V_1,\dots,V_{n-2}$ for the other vertices (ordered around the dial, so that $V_1$ and $V_{n-2}$ are neighbours of $P$). 

\begin{namedlemma}[Stopwatch Lemma]
\label{stopwatch}
For all $n\in\mathbb N$, $2\leq k 
\leq n$, $\operatorname{FS}(\mathcal S\mathcal W_n,\mathcal B_{n-k}^{(k)})$ is connected.
\end{namedlemma}

\begin{proof}
It is enough to show that all configurations with a social butterfly on $O$ can be reached from each other, seeing as we can always bring a social butterfly to $O$. 

Observe first that there is a sequence of swaps $\pi$ which fixes the social butterfly on $O$ and cyclically permutes the people on the dial vertices $V_1,V_2,\dots,V_{n-2}$ (including the social butterfly lying on $P$) by one position. This permutation $\pi$ can perhaps be seen more intuitively in the language of sliding blocks: the empty space allows us to slide all the non-empty blocks by one position along the dial (counter-clockwise for example). Note that in this scenario, the empty vertex itself is seen as moving clockwise $n-2$ positions around the dial. This motion can be repeated as many times as needed to obtain any power of $\pi$; we call this \emph{ticking} the stopwatch.

\begin{figure}[H]\centering
  \includegraphics[page=10]{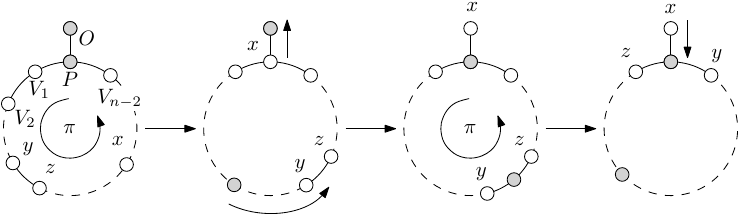}
  \caption{: The \nameref{stopwatch}: any person $x$ in the dial of a stopwatch can be removed and inserted back in the cycle between any other two persons $y$ and $z$ lying on successive vertices of the dial. To do so requires ``ticking'' the stopwatch the appropriate number of times, storing $x$ in the crown of the stopwatch, then ticking it some more before inserting the stored person $x$ back between $y$ and $z$.}
  \label{fig:clockwork}
  \end{figure}

For all other reconfigurations, the key observation is that we can store any person $x$ in the crown: after ticking the stopwatch to bring $x$ to $P$, we may swap it with the social butterfly at $O$. While doing this, we can position the other social butterfly left on the dial between any two people $y$ and $z$ lying on two chosen consecutive vertices of the dial. We then tick the stopwatch until this last social butterfly now lies on $P$, before swapping it with $x$ to insert it back between $z$ and $y$. Since there is once again a social butterfly on $O$, we can repeat this entire operation as many times as required, and all permutations can be reached by a sequence of such insertions via the \emph{insertion sort} algorithm (see for example \cite[5.2.1]{knuth}). 
\end{proof}
\begin{cor}
\label{cor}
For all $k\in \{2,3,\ldots, 7\}$, $\operatorname{FS}(\theta_0,\mathcal B_{7-k}^{(k)})$ is connected.
\end{cor}

\begin{proof}[Proof of \cref{cor}]
Notice that the $\theta_0$ graph can be seen as a stopwatch graph $\mathcal S \mathcal W_6$ with a single additional edge drawn between its crown and the unique vertex at distance $4$ from the crown. This additional edge only makes reconfiguration easier.
\end{proof}

Next, we say that a graph $\mathcal H$ is an \emph{hourglass} if it can be obtained by joining two cycles at a single vertex, called its \emph{throat}.

\begin{figure}[H]\centering
  \includegraphics[page=27]{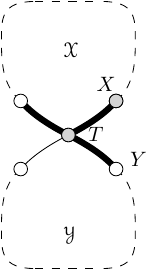}
  \caption{: An hourglass graph, consisting of two cycles joined at a single \emph{throat} vertex $T$. Its $\mathcal D_4$ snake tongue subgraph is highlighted in bold.}
  \label{fig:hourglass}
\end{figure}
\vspace{-1em}

\begin{namedlemma}[Hourglass Lemma]
\label{hourglass}
For all $n\in\mathbb N$, $2\leq k 
\leq n$, if $\mathcal H$ is an hourglass on $n$ vertices, then $\operatorname{FS}(\mathcal H,\mathcal B_{n-k}^{(k)})$ is connected.
\end{namedlemma}

\begin{proof}
    Let $\mathcal H$ be an hourglass graph on $n$ vertices and denote by $\mathcal X$ and $\mathcal Y$ its two cycles joined at its throat vertex $T$. Denote by $X$ (respectively $Y$) one of the two neighbours of $T$ in $\mathcal X$ (respectively $\mathcal Y$). We first note that there always lies a snake tongue $\mathcal D_4$ in the throat of the hourglass, with its unique degree 3 vertex coinciding with the throat vertex. Moreover, we may select it so that the (two-pointed) tip of the tongue lies in $\mathcal X$ and the rest of the tongue in $\mathcal Y$; see \cref{fig:hourglass}. By bringing two social butterflies to either $T$ and $X$ or $T$ and $Y$, we can operate the \nameref{stopwatch} in either of the stopwatches induced by $\mathcal X \cup Y$ or $\mathcal Y\cup X$. This allows us to bring any person $x$ in $\mathcal X$ (resp. $y$ in $\mathcal Y$) to a non-throat vertex of the snake tongue lying in $\mathcal X$ (resp. $\mathcal Y$). Using the \nameref{snaketongue} we may then exchange $x$ and $y$ across $T$ between $\mathcal X$ and $\mathcal Y$. By repeatedly applying this procedure, we can reach any desired partition of the people across $\mathcal X$ and $\mathcal Y$. An application of the \nameref{stopwatch} in both $\mathcal X$ and $\mathcal Y$ is then enough to reconfigure the people in each of the two cycles of the hourglass.
\end{proof}

We are now in a position to prove our last lemma, which establishes connectivity of $\operatorname{FS}(\mathcal G,\mathcal B_{n-k}^{(k)})$ for a much more general class of graphs called \emph{balloons}. A graph $\mathcal Q$ is a balloon if it can be obtained by adding a single leaf to a biconnected graph. Balloons can thus be seen as a generalisation of stopwatches. 

\begin{namedlemma}[Balloon Lemma]
\label{balloon}
Let $\mathcal Q$ be a balloon, obtained by adding a leaf $O$ to a vertex $P$ of a biconnected graph $\mathcal G$ on $(n-1)$ vertices. For all $k\geq 2$,  $\operatorname{FS}(\mathcal Q,\mathcal B_{n-k}^{(k)})$ is connected. 
\end{namedlemma}

On our way to prove the \nameref{balloon}, we will also prove the following variation of it, which will be used in the base case of the induction of the proof of our \nameref{thm:main}. This following lemma confirms our first impression that the addition of a second empty vertex/social butterfly grants us significantly more power --- recall that it allows one to connect all configurations of the $15$-puzzle. We show that more generally, two social butterflies are in fact enough to guarantee connectivity for all biconnected graphs except cycles (i.e.\ unlike in \nameref{thm:Wilson}, neither bipartite graphs nor the $\theta_0$ graph are problematic).

\begin{namedlemma}[Stringless Balloon Lemma]
\label{stringless}
Let $2 \leq k\leq n$ and consider a biconnected graph $\mathcal G$ on $n$ vertices which is not a cycle of length strictly greater than $k+1$. The friends and strangers graph $\operatorname{FS}(\mathcal G, \mathcal B_{n-k}^{(k)})$ is connected.
\end{namedlemma}

\begin{proof}[Proof of the \nameref{stringless}]

If $\mathcal G$ happens to be the $\theta_0$ graph, the lemma is already given by \cref{cor}. If $\mathcal G$ is a cycle on fewer than $k+1$ vertices, or if $\mathcal G$ consists of a single edge, the result is trivial. If $\mathcal G\ne \theta_0$ is not bipartite, then $\operatorname{FS}(\mathcal G, \mathcal B_{n-k}^{(k)})$ is connected for any positive $k$ and there is nothing further to prove. So, let $\mathcal G$ be a biconnected bipartite graph on at least three vertices which is not a cycle.

Recall the characterisation of the components of $\operatorname{FS}(\mathcal G, \mathcal B_{n-1}^{(1)})$ in case (2) of \nameref{thm:Wilson} (there are two connected components with an explicit description in terms of parity). Identifying once again $\mathtt V(\mathcal G)$ and $\mathtt V(\mathcal B_{n-k}^{(k)})$ with $\{1,2,\ldots, n\}$, it is enough to show that there exist configurations (seen as permutations in this context) $\sigma$ and $\sigma'$, reachable from each other, which fix a social butterfly on some arbitrary vertex $X$ of $\mathcal G$ but have opposite parities. Because the \nameref{stopwatch} allows full reconfiguration of the stopwatch, and thus allows one to obtain a permutation of arbitrary signature of its dial, it is enough to see that $\mathcal G$ contains a stopwatch. Indeed, $\mathcal G$ is biconnected with at least 3 vertices and thus contains a cycle. Consider a shortest such cycle $\mathcal C$. Since $\mathcal G\ne \mathcal C$ is connected, there is an additional edge incident to $\mathcal C$, with which we may form a stopwatch. 
\end{proof}

\begin{figure}[H]\centering
\includegraphics[page=26]{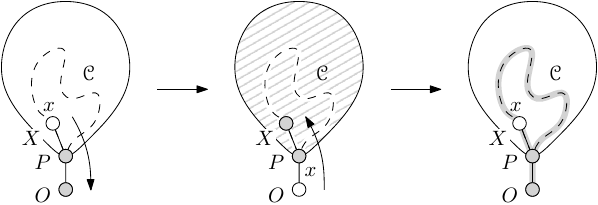}
\caption{: The sequence of reconfigurations from $\sigma$ to $\sigma'$. The grey hatchings indicate the completion of the same configuration as the target configuration $\sigma'$, restricted to the hatched area.}
\label{fig:balloon}
\end{figure}
\vspace{-1em}

\begin{proof}[Proof of the \nameref{balloon}]
Let $\mathcal B$ be a balloon, obtained by adding a leaf $O$ to a vertex $P$ of a biconnected graph $\mathcal G$ on $(n-1)$ vertices. The case where $\mathcal G$ is a cycle is already treated by the the \nameref{stopwatch}, and the case where $\mathcal G$ is an edge is trivial. Suppose then that $\mathcal G$ is a biconnected graph which is not a cycle. We consider two configurations $\sigma$ and $\sigma'$ of $\mathcal B$ with social butterflies positioned on $O$ and $P$, and show that these configurations can be reached from each other (it is always easy to bring two social butterflies to $O$ and $P$).

Since $\mathcal G$ is biconnected, it has a cycle $\mathcal C$ containing $P$. Let us call $X$ one of the two neighbours of $P$ on $\mathcal C$ and call $x$ the person lying on $X$ in the configuration $\sigma$. We begin by swapping $x$ successively with the two social butterflies lying on $P$ and $O$, so that $x$ is now lying on $O$ and the two social butterflies on $P$ and $X$. We may then apply the \nameref{stringless} to reconfigure $\mathcal G$ as we wish, leaving $O$ and $P$ fixed. In particular, we may choose to reach a configuration $\sigma_1$ such that $\sigma_1$ and $\sigma'$ agree on everything but $\mathcal C$, and share the same set of people lying on $\mathcal C$. We then swap back $x$ onto $X$, bringing back the two social butterflies to rest on $O$ and $P$. To finish the proof, we then apply the \nameref{stopwatch} to reconfigure the stopwatch obtained by adding the leaf $O$ at $P$ to the cycle $\mathcal C$ to match ${\sigma'}$ on~$\mathcal{C}$.
\end{proof}

\section{Proof of the Main Theorem}
\label{main}

\subsection*{Rough Sketch of the Proof}\label{sketch} We first recall that $\operatorname{FS}(\mathcal G,\mathcal B_{n-k}^{(k)})$ is disconnected if there exist dark hallways in $\mathcal G$ (\cref{DF-hallways}). This observation combined with \nameref{thm:Wilson} yields a lower bound of $\kappa^*(\mathcal G)$ on the smallest $k$ such that $\operatorname{FS}(\mathcal G,\mathcal B_{n-k}^{(k)})$ is connected. We may then focus on proving the upper bound. Note also that if $\mathcal G$ is not connected, neither is $\operatorname{FS}(\mathcal G,\mathcal B_{n-k}^{(k)})$ (hence the infinite value of $\kappa^*$) and we may thus restrict our attention to connected graphs.

As outlined in \cref{overview}, we wish to bootstrap our analysis of each library. We do this by induction on the number of blocks.

The structure of our proof is as follows. We first isolate a block of $\mathcal G$ corresponding to a leaf of its block graph. Let us call this block $\mathcal L$. Since blocks intersect in unique cut-vertices, $\mathcal L$ intersects its neighbouring block in a single cut-vertex which we denote by $V$. Let us call $\mathcal A$ the subgraph of $\mathcal G$ induced by the union of the complement of $\mathtt V(\mathcal L)$ with $V$. (So, $\mathcal A$ and $\mathcal L$ partition the edges of $\mathcal G$, and they share a single vertex $V$). Our proof then consists of two main ingredients. First, we need to show that we can ferry people across $V$ between $\mathcal A$ and $\mathcal L$ to reach any desired ``people-partition'' between $\mathcal L$ and $\mathcal A$ (the \nameref{B}). More specifically, we will single out a special type of partition which we call \emph{balanced}; such partitions roughly speaking have the social butterflies balanced between the two sides $\mathcal A$ and $\mathcal L$ in a way that facilitates induction (the exact definition will be given shortly thereafter). Second, we need to show that all the configurations which feature the same balanced partition can be reached from one another (the \nameref{A}).

\begin{figure}\centering
  \includegraphics[page=47]{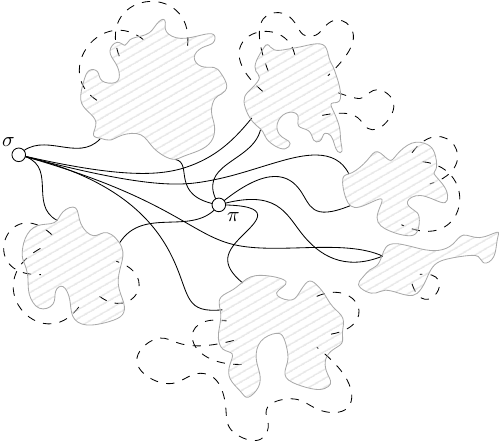}
  \caption{: A pictorial overview of the proof of the main theorem: configurations sharing the same balanced partition form ``islands'' that are connected to one another in $\operatorname{FS}(\mathcal G,\mathcal B_{n-k}^{(k)})$ (hatched in grey). Every non-balanced configuration 
(such as $\sigma$ or $\pi$ in the figure) has a path (a ``ferry'') to all these connected configurations.}
  \label{fig:overview}
\end{figure}

Geometrically, our two steps correspond to showing the existence of subgraphs of $\operatorname{FS}(\mathcal G,\mathcal B_{n-k}^{(k)})$ such that every pair of vertices in these subgraphs can be reached from one another --- note that this might involve leaving the subgraph --- and spelling out a path from any configuration $\sigma$ to any of these subgraphs. We think of our subgraphs as ``islands'', one for each balanced partition, and we think of our paths between islands as ``ferries''; see \cref{fig:overview}. 

\subsection*{The Base Cases}\label{base} Because of the peculiar role played by cycles in our theorem, we need two different base cases for our induction:
\begin{itemize}
\item \textit{Case 1: $\mathcal G$ is an hourglass (which has two blocks).} In this case, $\kappa(\mathcal G)=1$, so $\kappa^*(\mathcal G)=2$, and the \nameref{thm:main} follows from the \nameref{hourglass}.
\item \textit{Case 2: $\mathcal G$ consists of a single block (i.e.\ it is biconnected).} In this case $\kappa(\mathcal G)=1$, but $\kappa^*(\mathcal G)$ depends on the structure of $\mathcal G$.
\begin{itemize}
\item If $\mathcal G\ne \theta_0$ is non-bipartite then $\kappa^*(\mathcal G)=1$ and the desired result follows from \nameref{thm:Wilson}.
\item If $\mathcal G= \theta_0$ or $\mathcal G$ is bipartite then $\kappa^*(\mathcal G)=2$. At least two social butterflies are necessary for connectedness by \nameref{thm:Wilson}, and two social butterflies suffice by the \nameref{stringless}.
\item If $\mathcal G$ is a cycle then $\kappa^*(\mathcal G)=n-1$. It is impossible to change the relative order of the asocial people around the cycle, so $\operatorname{FS}(\mathcal G,\mathcal B_{n-k}^{(k)})$ is connected if and only if $k\ge n-1$ (i.e.\ there is at most one asocial person).
\end{itemize}
\end{itemize}

\subsection*{The Induction Step} 
Now, fix a connected graph $\mathcal G$ on $n$ vertices, which is not an hourglass and has at least two blocks. Suppose that the \nameref{thm:main} holds for all subgraphs of $\mathcal G$ with fewer blocks than $\mathcal G$. We will prove the \nameref{thm:main} for $\mathcal G$.

Let us fix $k\leq n$ and assume that $\mathcal G$ does not contain any dark hallways. Our goal is to prove that $\operatorname{FS}(\mathcal G,\mathcal B_{n-k}^{(k)})$ is connected as in the \nameref{sketch}. Fix a leaf $\mathcal L$ of the block graph of $\mathcal G$, and let $\mathcal A$ be subgraph induced by the union of the complement of $\mathcal L$ with $V$, where $V$ is the unique cut-vertex separating $\mathcal A$ from $\mathcal L$. Since $\mathcal G$ is not an hourglass, we may choose the leaf $\mathcal L$ in such a way that $\mathcal A$ is not a cycle.

We define a \textit{block party} to be a partition $(\mathcal X,\mathcal Y)$ of $\mathtt V(\mathcal B_{n-k}^{(k)})$ into two sets of size $|\mathcal X|=|\mathcal A|$ and $|\mathcal Y|=| \mathcal L|-1$. For a block party $(\mathcal X,\mathcal Y)$ we then say that a configuration $\sigma$ is $(\mathcal X,\mathcal Y)$-\textit{balanced} if the following conditions hold:

\begin{enumerate}
    \item The people of $\mathcal X$ are the ones lying in $\mathcal A$, i.e.\ $\sigma(\mathtt V(\mathcal A))=\mathcal X$.
    \item The people of $\mathcal Y$ are the ones lying in $\mathcal L-\{V\}$, i.e.\ $\sigma(\mathtt V(\mathcal L) -\{V\})=\mathcal Y$.
    \item There is a social butterfly lying on the cut-vertex $V$ and on a vertex of $\mathcal A$ adjacent to $V$.
    \item Either there are no social butterflies in $\mathcal Y$, or $\mathcal X$ consists entirely of social butterflies (roughly speaking, we want all the social butterflies to be in $\mathcal A$, but if there are more than $|\mathcal A|$ social butterflies we allow some of them to overflow into $\mathcal Y$).
\end{enumerate}

If $(\mathcal X,\mathcal Y)$ has the right number of social butterflies in $\mathcal X$ and in $\mathcal Y$ for a $(\mathcal X,\mathcal Y)$-balanced configuration to exist, then we say $(\mathcal X,\mathcal Y)$ is a \emph{balanced block party}. Furthermore, we can associate to every configuration $\sigma$ a unique block party $(\mathcal X_{\sigma}, \mathcal Y_{\sigma}) \coloneqq (\sigma(\mathtt V(\mathcal A)), \sigma(\mathtt V(\mathcal L)-\{V\}) )$. For each balanced block party $(\mathcal X,\mathcal Y)$, we define the $(\mathcal X,\mathcal Y)$-\emph{cluster} to be the set of configurations whose associated block party is $(\mathcal X,\mathcal Y)$.

As outlined in the \nameref{sketch} at the start of this section, there are two main ingredients to our proof.

\begin{namedlemma}[Island Lemma]
  \label{A}
  For every balanced block party $\left(\mathcal X, \mathcal Y \right)$, all the configurations in the $(\mathcal X,\mathcal Y)$-cluster are part of the same connected component of $\operatorname{FS}(\mathcal G,\mathcal B_{n-k}^{(k)})$.
\end{namedlemma}

\begin{namedlemma}[Ferrying Lemma]
  \label{B}
  For every configuration $\sigma$, and for every balanced block party $\left(\mathcal X, \mathcal Y \right)$, there exists a path in $\operatorname{FS}(\mathcal G,\mathcal B_{n-k}^{(k)})$ between $\sigma$ and the $(\mathcal X,\mathcal Y)$-cluster.
\end{namedlemma}

\begin{proof}[Proof of the \nameref{A}]
Let $\sigma$ and $\sigma'$ be two $(\mathcal X,\mathcal Y)$-balanced configurations. Our aim is to construct a sequence of swaps from $\sigma$ to $\sigma'$. 

\begin{figure}[H]\centering
\includegraphics[page=38]{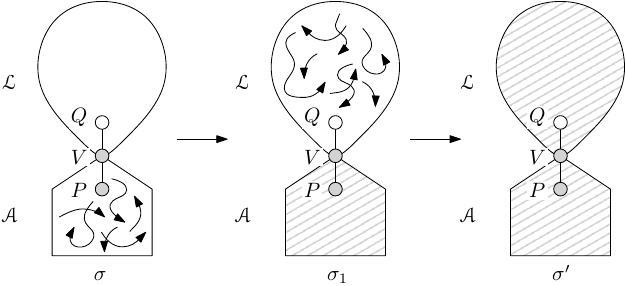}
\caption{: The sequence of reconfigurations from $\sigma$ to $\sigma'$. The hatchings indicate the completion of the same configuration as the target configuration $\sigma'$, restricted to the hatched area.}
\label{fig:claim-1-2}
\end{figure}
\vspace{-1em}

By assumption, $\mathcal G$ does not have any dark hallways, and thus neither does $\mathcal A$. Since $\mathcal A$ is not a cycle, we may use our induction assumption to reconfigure $\mathcal A$ in isolation to reach a $(\mathcal X,\mathcal Y)$-balanced configuration $\sigma_1$ which agrees with $\sigma'$ on $\mathcal A$ (i.e.\ ${\sigma_1}|_{\mathcal A}={\sigma'}|_{\mathcal A}$). Since $\sigma_1$ is balanced, we can select a neighbour $P$ of $V$ in $\mathcal A$ supporting a social butterfly. The subgraph induced by the union of $P$ with $\mathcal L$ forms a balloon and two social butterflies lie on $P$ and $V$. Thus, we may apply the \nameref{balloon} to reconfigure $\mathcal L-\{V\}$ as needed and reach $\sigma'$.
\end{proof}

We have now verified that all configurations lying inside a cluster of balanced configurations corresponding to a fixed balanced block party are connected to one another. To complete the proof, we exhibit a path from any configuration to any balanced cluster.

\begin{proof}[Proof of the \nameref{B}]
Let us fix a configuration $\sigma$ and its associated block party $(\mathcal X_{\sigma}, \mathcal Y_{\sigma})$, together with the target balanced block party $(\mathcal X,\mathcal Y)$. We begin by successively bringing as many as possible of $k$ available social butterflies to $\mathcal A$, thereby reaching an $(\mathcal X',\mathcal Y')$-balanced configuration $\sigma'$ for some balanced block party $(\mathcal X',\mathcal Y')$. We want to show that for any asocial people $x\in \mathcal X$ and $y\in \mathcal Y$, there is always an sequence of swaps which allows us to alter $(\mathcal X',\mathcal Y')$ by ``trading $x$ with $y$'', as follows.

\begin{claim*}
\label{trade}
For all asocial $x\in \mathcal X'$ and $y\in \mathcal Y'$, there exists a sequence of swaps which takes $\sigma'$ to some $(\mathcal X'\cup\{y\}-\{x\}, \mathcal Y'\cup\{x\}-\{y\})$-balanced configuration.
\end{claim*}

The above claim suffices to prove the \nameref{A}: indeed, we can reach $(\mathcal X,\mathcal Y)$ from $(\mathcal X',\mathcal Y')$ by repeatedly trading people between the two parts.

\begin{proof}[Proof of the claim]
First, note that if $\mathcal A$ is a path, then the assumption that there do not exist any dark hallways forces $\mathcal A$ to be a path of length at most $k$, and the desired result is trivial (in the $(\mathcal X',\mathcal Y')$-balanced configuration $\sigma'$, only social butterflies lie in $\mathcal A$).

So, we may assume $\mathcal A$ is not a path. Similarly to the proof of the \nameref{hourglass}, we find an adequately short snake tongue subgraph in $\mathcal G$, straddling $V$ and extending into both $\mathcal A$ and $\mathcal L$ (see \cref{fig:claim-3-2}), so that we may use the \nameref{snaketongue} to ferry people across $V$. To be precise, the fact we need is as follows.

\begin{fact*}
There always exists a snake tongue subgraph $\mathcal D_m$, $4 \leq m \leq k+1$, such that the entirety of $\mathcal D_m$ is contained in $\mathcal A$, save for the very last edge of the ``back'' of the tongue which lies in $\mathcal L$.
\end{fact*}

\begin{figure}[H]\centering
\includegraphics[page=46]{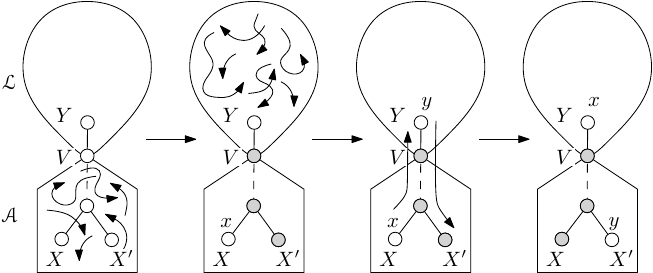}
\caption{: The sequence of swaps to alter $(\mathcal X', \mathcal Y')$ by trading $x$ with $y$.}
\label{fig:claim-3-2}
\end{figure}
\vspace{-1em}

This snake tongue can be found by exploring from $V$ into $\mathcal A$. If $V$ has two neighbours in $\mathcal A$ we have immediately found a suitable snake tongue $\mathcal D_4$, and otherwise the desired snake tongue amounts to a vertex of degree $d\geq 3$ in $\mathtt V(\mathcal A)\cup \{V\}$ within distance at most $k-1$ of $V$. If we cannot find such a snake tongue, then the depth-$(k-1)$ neighbourhood of $\mathcal V$ in $\mathcal A$ must be a path. Since $\mathcal A$ has no dark hallways, this means $\mathcal A$ itself must be a path, contradicting the assumption we made at the start of the proof of the claim.

Proceeding with the proof of the claim, let us write $X$ and $X'$ for the two vertices at the tip of the snake tongue graph $D_m$, and $Y$ for its back vertex (see \cref{fig:claim-3-2}). Since $\mathcal A$ is not a cycle, our induction assumption guarantees that we may reconfigure $\mathcal A$ as desired. In particular, we may bring any person $x\in \mathcal X'$ to $X$ and saturate the rest of $D_m\cap \mathcal A$ with $m-1\leq k$ social butterflies. Note that $V$ has a neighbour $X''\neq X$ in $D_m$ (if $m=4$ we take $X''=X'$, and otherwise we let $X'$ be a vertex on the path section of the snake tongue). The subgraph induced by $X''$ and $\mathcal L$ is a balloon, and two social butterflies now lie on $X''$ and $V$. We may thus apply the \nameref{balloon} in $\mathcal L$ to bring any person $y\in \mathcal Y'$ to $Y$. The \nameref{snaketongue} then allows us to reconfigure $D_m$ so that $y$ now lies on $X$ and $x$ on $Y$. This final configuration $\sigma_{x,y}$ satisfies the conclusion of the claim.
\end{proof}
\end{proof}

\section{A simple algorithm to determine when two configurations are connected}
\label{connected components}

The ideas in the proof of our \nameref{thm:main} can be straightforwardly adapted to provide an algorithm that determines whether two configurations can be reached from each other. 

\begin{figure}[H]\centering
\includegraphics[page=36]{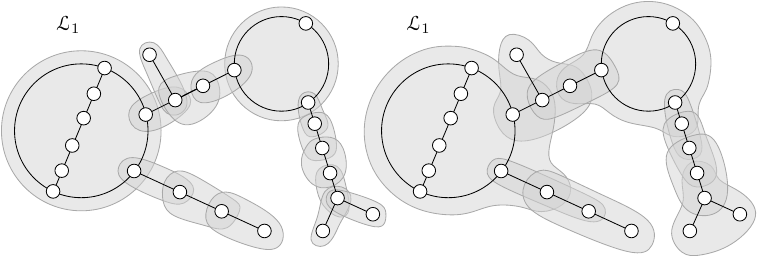}
\caption{: The social circles of the graph $\mathcal G$ for $k=1$ (left diagram) and $k=2$ (right diagram).}
\label{fig:social-circles-1}
\end{figure}
\vspace{-1em}

Given a connected graph $\mathcal G$ (and some $k$), a connected subgraph $\mathcal G'$ of $\mathcal G$ is called \emph{well-lit} if it does not contain any dark hallways (for the subgraph). A maximal well-lit subgraph is called a \emph{social circle} (see \cref{fig:social-circles-1}). By construction, the social circles of a graph $\mathcal G$ cover all of its vertices. It is easy to see that all social circles contain at least $k$ vertices (in fact, unless $k=n$, each social circle contains at least $k+1$ vertices, but we will not need this fact).

Setting aside issues related to the exceptions in \nameref{thm:Wilson}, our \nameref{thm:main} allows one to freely reconfigure inside each social circle. Thus, to check that two configurations can be reached from each other, we need only verify that each social circle contains the same set of people in both configurations, and further consider the conditions of \nameref{thm:Wilson} within each social circle.

To be more precise, consider the social circles $\mathcal L_1, \mathcal L_2, \ldots, \mathcal L_m$ of $\mathcal G$. For all $i\in [m]$, we arbitrarily fix a set $\mathcal O_i$ of $k$ vertices in the social circle $\mathcal L_i$. The key observation is that if all $k$ social butterflies lie on the vertices in $\mathcal O_i$, then none of the other asocial people who are in $\mathcal L_i$ at that point can ever leave $\mathcal L_i$ (since asocial people cannot traverse dark hallways). That is to say, if we have two configurations $\sigma_i,\tau_i$ for which the $k$ social butterflies lie in $\mathcal O_i$, then we have no hope of reaching $\tau_i$ from $\sigma_i$ unless $\mathcal L_i$ contains the same set of people in $\sigma_i$ and $\tau_i$. But note that for \emph{any} configurations $\sigma,\tau$, we can always reach some configuration $\sigma_i$ from $\sigma$, and some configuration $\tau_i$ from $\tau$, such that $\sigma_i,\tau_i$ have the $k$ social butterflies lying in $\mathcal O_i$ (indeed, social butterflies can move freely through the graph, displacing asocial people as they move).

Below we present an algorithm applying these observations to check if two configurations can be reached from each other. For ease of notation, given a social circle $\mathcal L$ in a graph $\mathcal G$, and some $\sigma\in \mathtt V(\operatorname{FS}(\mathcal G,\mathcal B_{n-k}^{(k)}))$, we write $\sigma|_\mathcal L$ for the restriction of $\sigma$ to the vertices in $\mathcal L$. If we write $n_{\mathcal L}$ and $k_{\sigma,\mathcal L}$ for the number of vertices in $\mathcal L$ and the number of social butterflies in $\mathcal L$ with respect to $\sigma$, then $\sigma(\mathtt V(\mathcal L))$ is a set of $n_{\mathcal L}$ people including $k_{\sigma,\mathcal L}$ social butterflies. We can therefore identify $\sigma(\mathtt V(\mathcal L))$ with the vertex set of $\mathcal B_{n_{\mathcal L}-k_{\sigma,\mathcal L}}^{(k_{\sigma,\mathcal L})}$, and view $\sigma|_{\mathcal L}$ as a vertex of $\operatorname{FS}(\mathcal L,\mathcal B_{n_{\mathcal L}-k_{\sigma,\mathcal L}}^{(k_{\sigma,\mathcal L})})$.

\begin{algorithm}[H]
	\caption*{\textbf{Connect}\boldmath $\pmb(\sigma, \tau,\mathcal G, k \pmb )$ Checking if two configurations $\sigma$ and $\tau$ are connected} 
	\begin{algorithmic}
        \State Consider the social circles $\mathcal L_1, \mathcal L_2, \ldots, \mathcal L_m$ of $\mathcal G$.
        \State Fix a set $\mathcal O_i$ of $k$ vertices in each social circle $\mathcal L_i$.
		\For {$i=1,2,\ldots, m$}
            \State \parbox[t]{313pt}{In both $\sigma$ and $\tau$, move all social butterflies to $\mathcal O_i$ (arbitrarily) and consider the resulting configurations $\sigma_i,\tau_i$.}
            \vspace{0.2em}
            \State Run \textbf{SocialCircle}\boldmath $\pmb({\sigma_i}, {\tau_i}, {\mathcal L_i}, {k}\pmb)$
        \EndFor
        \State Return ``CONNECTED''
	\end{algorithmic} 
\end{algorithm}

\begin{algorithm}[H]
	\caption*{\textbf{SocialCircle}\boldmath $\pmb(\sigma_i, \tau_i,\mathcal L_i, k \pmb )$ Checking connectivity in the social circle $\mathcal L_i$} 
	\begin{algorithmic}
            \If{$\sigma_i(\mathtt V(\mathcal L_i))\neq\tau_i(\mathtt V(\mathcal L_i))$}
                \State Return ``NOT CONNECTED''
            \EndIf
        \If{$k < |\mathcal L_i|-1$ and $\mathcal L_i$ is a cycle}
            \If{the cyclic ordering of $\sigma_i$ and $\tau_i$ differ}
                \State Return ``NOT CONNECTED''
            \EndIf
            \EndIf

            \If{$k=1$}
            \If{$\mathcal L_i$ is bipartite}
            \If{$\sigma_i$ and $\tau_i$ do not share the same parity}
                \State Return ``NOT CONNECTED''
            \EndIf
        \ElsIf{$\mathcal L_i\cong\theta_0$}
            
             \hspace{1em}\parbox[t]{313pt}{If we have reached this point then $\sigma_i(\mathtt V(\mathcal L_i))=\tau_i(\mathtt V(\mathcal L_i))$.  Note that $\sigma_i|_{\mathcal L_i}$ and $\tau_i|_{\mathcal L_i}$ can then be viewed as vertices of $\operatorname{FS}(\mathcal L_i,B^{(1)}_{7-1})\cong \operatorname{FS}(\theta_0,B^{(1)}_{7-1})$, which is a specific graph on $7!$ vertices with $6$ connected components (characterised in \cite{Wilson}).\strut}
            \If{$\sigma_i|_{\mathcal L_i}$ and $\tau_i|_{\mathcal L_i}$ are in different components of $\operatorname{FS}(\mathcal L_i,B^{(1)}_{7-1})$}
                \State Return ``NOT CONNECTED''
            \EndIf
        \EndIf
        \EndIf
	\end{algorithmic} 
\end{algorithm}

We note that the above algorithm is efficient (i.e.\ it terminates in polynomial time). Let $n,e,m$ be the number of vertices, edges and social circles of $\mathcal G$. First, we note that the blocks and cut-vertices of $\mathcal G$ can be found in time $O(n+e)$ (\cite{paton}), after which the social circles can be returned in time $O(kn)$ (remember that the block graph is always a forest). The operation of ``moving'' the social butterflies to $\mathcal O_i$ requires computing paths between all social butterflies and social circles, which can be done in time $O(\min\{m,k\}(n+e))$ by either conflating $k$ points inside each social circle to a single node and running a Breadth-First-Search rooted at that node, or running a BFS directly from each social butterfly. Regarding \textbf{SocialCircle}, all the criteria take $O(|\mathcal L_i|)$ time to compute. Since $\sum_i |\mathcal L_i|<kn$, the overall running time of our algorithm is $O(\min\{m,k\}(n+e)+2kn)$. Note that this is $O(n^3)$ in the worst case (and the cubic lower bound of \cite{Thesis} still holds), but much faster on many graphs. 

\section{The number of components of the friends and strangers graph} 
\label{sec:structure}
The previous algorithm not only informs us on the connectivity between two particular configurations, but also shows the following:

\begin{cor}
\label{formula}
The number of connected components of $\operatorname{FS}(\mathcal G, \mathcal B_{n-k}^{(k)})$ is precisely

\[\binom{n-k}{|\mathcal L_1|-k,|\mathcal L_2|-k,\ldots, |\mathcal L_m|-k}\cdot \prod_{i=1}^m \lambda_i,\]

where $\lambda_i$ is defined to be equal to $1$, except in the following cases:

\begin{itemize}[label=\raisebox{0.25ex}{\tiny$\bullet$}]
    \item If $k < |\mathcal L_i|-1$ and $\mathcal L_i$ is a cycle, $\lambda_i\coloneqq \frac{|\mathcal L_i|!}{2}$.
    \item If $k=1$, then:
    \begin{itemize}
        \item If $\mathcal L_i$ is bipartite, but not a cycle, $\lambda_i\coloneqq 2$.
        \item If $\mathcal L_i\equiv \theta_0$, $\lambda_i\coloneqq 6$.
    \end{itemize}
\end{itemize}
\end{cor}
Indeed, the collection of asocial people lying in $\mathcal L_i$ in the configuration $\sigma_i$, as $i$ ranges over $[m]$, forms a partition of all asocial people. As shown by our previous algorithm, two configurations can be reached from one another if and only if they agree on this partition and obey the additional constraints of \nameref{thm:Wilson}. This confirms that the previous formula is both a lower and an upper bound for the number of connected components.


\bibliography{references.bib}
\bibliographystyle{plain}

\end{document}